\documentclass[12pt,twoside]{amsart}
\usepackage{amssymb,amsmath,amsthm, amscd, enumerate, mathrsfs}
\usepackage{graphicx, hhline}
\usepackage[all]{xy}

\title[A relative spannedness]{A relative 
spannedness for log canonical pairs and 
quasi-log canonical pairs}
\author{Osamu Fujino}
\date{2020/11/29, version 0.15}
\address{Department of 
Mathematics, Graduate School of Science, 
Osaka University, Toyonaka, Osaka 560-0043, Japan}
\email{fujino@math.sci.osaka-u.ac.jp}

\DeclareMathOperator{\Nqlc}{Nqlc}
\DeclareMathOperator{\Nklt}{Nklt}
\DeclareMathOperator{\Nlc}{Nlc}
\DeclareMathOperator{\NLC}{NLC}

\DeclareMathOperator{\Bs}{Bs}
\DeclareMathOperator{\Coker}{Coker}
\DeclareMathOperator{\Supp}{Supp}
\newtheorem{thm}{Theorem}[section]
\newtheorem{lem}[thm]{Lemma}
\newtheorem{prop}[thm]{Proposition}
\newtheorem{cor}[thm]{Corollary}
\newtheorem*{claim}{Claim}

\theoremstyle{definition}
\newtheorem{step}{Step}
\newtheorem{ex}[thm]{Example}
\newtheorem{defn}[thm]{Definition}
\newtheorem{rem}[thm]{Remark}
\newtheorem*{ack}{Acknowledgments}  

\makeatletter
    
    \@addtoreset{equation}{section}
\makeatother

\begin{document}

\maketitle 

\begin{abstract}
We establish a relative spannedness for log canonical 
pairs, which is a generalization of the basepoint-freeness for 
varieties with log-terminal singularities by Andreatta--Wi\'sniewski. 
Moreover, we establish a generalization for quasi-log canonical 
pairs.  
\end{abstract}

\tableofcontents

\section{Introduction}\label{f-sec1} 

The main purpose of this paper is to establish the 
following relative spannedness for log canonical pairs. 

\begin{thm}[{Relative spannedness for log canonical 
pairs, see \cite
[Theorem, Remark 3.1.2, and 
Theorem 5.1]{andreatta-w}}]\label{f-thm1.1}
Let $(X, \Delta)$ be a log canonical 
pair and let $f\colon X\to Y$ be a projective surjective morphism onto a variety 
$Y$ such that $-(K_X+\Delta)$ is $f$-ample. 
Let $L$ be a Cartier divisor on $X$. 
Assume that $K_X+\Delta+rL$ is relatively numerically trivial 
over $Y$ for some positive real number $r$. 
Let $F$ be a fiber of $f$. 
Then the dimension of every positive-dimensional 
irreducible component of 
$F$ is $\geq r-1$. We further assume that $\dim F<r+1$. 
Then 
$
f^*f_*\mathcal O_X(L)\to \mathcal O_X(L)
$
is surjective at every point of $F$. 
\end{thm}

As an easy consequence of Theorem \ref{f-thm1.1}, we have: 

\begin{cor}[{see \cite[Theorem 1]{fujita-polarized}}]\label{f-cor1.2}
Let $(X, \Delta)$ be a log canonical 
pair with $\dim X=n$ and let $f\colon X\to Y$ be a projective 
morphism onto a variety $Y$. 
Let $L$ be an $f$-ample Cartier divisor on $X$. 
Then $K_X+\Delta+(n+1)L$ is $f$-nef. 
Moreover, if $\dim Y\geq 1$, then 
$K_X+\Delta+nL$ is $f$-nef. 
\end{cor}

By Theorem \ref{f-thm1.1}, we can quickly recover 
the basepoint-freeness obtained by Andreatta and Wi\'sniewski 
in \cite{andreatta-w}. 

\begin{cor}[{Relative spannedness for kawamata log terminal 
pairs, see \cite
[Theorem, Remark 3.1.2, and 
Theorem 5.1]{andreatta-w}}]\label{p-cor1.3}
In Theorem \ref{f-thm1.1}, 
we further assume that $(X, \Delta)$ is kawamata log terminal 
and that $\dim X=\dim Y$. 
Then the dimension of every positive-dimensional 
irreducible component of $F$ 
is $\geq \lfloor r\rfloor$. 
Moreover, if $\dim F\leq r+1$, then 
$
f^*f_*\mathcal O_X(L)\to \mathcal O_X(L)
$
is surjective at every point of $F$. 
\end{cor}

The following easy example shows that the estimates 
on the lower bound of the dimension of the fiber are sharp. 

\begin{ex}\label{f-ex1.4}
Let $Y$ be a smooth variety 
with $\dim Y=n$. 
Let $f\colon X\to Y$ be a blow-up at a smooth 
point of $Y$ and let $E\simeq \mathbb P^{n-1}$ 
be the $f$-exceptional divisor on $X$. 
In this situation, $L:=-E$ is an $f$-ample Cartier divisor 
on $X$. 
We put $\Delta=E$. 
Then we obtain that $(X, \Delta)$ is log canonical 
and is not kawamata log terminal and that 
$K_X+\Delta+nL=f^*K_Y$ holds. 
\end{ex}

The following example shows that the assumption 
$\dim F<r+1$ in Theorem \ref{f-thm1.1} is sharp. 

\begin{ex}\label{f-ex1.5} 
Let $S$ be a Del Pezzo surface of degree one, that is, 
$(-K_S)^2=1$. 
We can easily check that 
$\dim _{\mathbb C}H^0(S, \mathcal O_S(-K_S))=2$ and 
that $\Bs\!|-K_S|$ is a point. 
In particular, $|-K_S|$ is not basepoint-free. We take 
a positive integer $m$ such that 
$$
\Phi_{|-mK_S|}\colon S\hookrightarrow \mathbb P^N
$$ 
is a projectively normal embedding. Let $Y\subset \mathbb A^{N+1}$ 
be the cone over $S$. Then $Y$ has only kawamata log terminal singularities. 
Let $f\colon X\to Y$ be the blow-up at the vertex $P\in Y$ and let 
$F\simeq S$ be the exceptional divisor of $f$. 
We put $\Delta=F$. 
Then $X$ is smooth, $(X, \Delta)$ is log canonical and 
is not kawamata log terminal, 
and $-(K_X+\Delta)$ is $f$-ample. 
We put $L=-(K_X+\Delta)$. 
Then $K_X+\Delta+rL$ with $r=1$ is obviously 
relatively numerically trivial over $Y$. 
We note that $\dim F=\dim S=2=r+1$. 
By adjunction, we have $L|_F=-K_F$. 
Since $F\simeq S$, $|L|_F|$ is not basepoint-free. 
This implies that 
$$
f^*f_*\mathcal O_X(L)\to \mathcal O_X(L)
$$ 
is not surjective at some point of $F$. 
\end{ex}

\medskip 

The original proof of Theorem \ref{f-thm1.1} 
and Corollary \ref{p-cor1.3} for varieties with only kawamata log terminal 
singularities in 
\cite{andreatta-w} is based on Koll\'ar's modified basepoint-freeness 
method in \cite{kollar-effective}. 
Although Koll\'ar's method was already generalized for 
log canonical pairs and quasi-log canonical pairs 
(see \cite{fujino-effective} and \cite{fujino-kollar-type}), 
we do not use it in this paper.  
Our proof of Theorem \ref{f-thm1.1} and 
Corollary \ref{p-cor1.3} heavily depends on 
the following basepoint-free theorem 
for projective quasi-log schemes. 

\begin{thm}[Spannedness for projective quasi-log schemes]\label{f-thm1.6}
Let $[X, \omega]$ be a projective quasi-log scheme 
and let $X_{-\infty}$ denote the non-qlc locus of $[X, \omega]$. 
We assume $\dim (X\setminus X_{-\infty})=n$. 
Let $\mathcal L$ be an ample line bundle on $X$ such that 
$\omega+r\mathcal L$ is numerically trivial with $r>n-1$. 
We further assume that $|\mathcal L|_{X_{-\infty}}|$ is basepoint-free. 
Then the complete linear system $|\mathcal L|$ is basepoint-free. 
\end{thm}

We prove Theorem \ref{f-thm1.6} by the 
theory of quasi-log schemes with the aid of Fujita's 
theory of $\Delta$-genera (see \cite{fujita1}). 
Then Theorem \ref{f-thm1.1} will be proved with an inductive argument 
via Theorem \ref{f-thm1.6}. 

We can further generalize Theorem \ref{f-thm1.1} 
for quasi-log canonical pairs. 
The precise statement is as follows: 

\begin{thm}[Relative spannedness for quasi-log canonical 
pairs]\label{f-thm1.7}
Let $[X, \omega]$ be a quasi-log canonical 
pair and let $\varphi\colon X\to W$ be a projective surjective 
morphism onto a scheme 
$W$ such that $-\omega$ is $\varphi$-ample. 
Let $\mathcal L$ be a line bundle on $X$. 
Assume that $\omega+r\mathcal L$ is relatively numerically trivial 
over $W$ for some positive real number $r$. 
Let $F$ be a fiber of $f$. 
Then the dimension of every positive-dimensional 
irreducible component of 
$F$ is $\geq r-1$. We further assume that $\dim F<r+1$. 
Then 
$
\varphi^*\varphi_*\mathcal L\to \mathcal L
$
is surjective at every point of $F$. 
\end{thm}

As a corollary of Theorem \ref{f-thm1.7}, we have the following 
generalization of Corollary \ref{f-cor1.2}. 

\begin{cor}\label{f-cor1.8}
Let $[X, \omega]$ be a quasi-log canonical 
pair with $\dim X=n$ and let $\varphi\colon X\to W$ be a projective 
morphism onto a scheme $W$. 
Let $\mathcal L$ be a $\varphi$-ample line bundle on $X$. 
Then $\omega+(n+1)\mathcal L$ is $\varphi$-nef. 
We further assume that $X$ is irreducible and 
$\dim W\geq 1$. Then 
$\omega+n\mathcal L$ is $\varphi$-nef. 
\end{cor}

Since every quasi-projective semi-log canonical pair naturally becomes 
a quasi-log canonical pair by \cite[Theorem 1.1]{fujino-slc}, 
we can apply Theorem \ref{f-thm1.7} and Corollary \ref{f-cor1.8} 
to semi-log canonical pairs. 

\medskip 

We briefly explain the organization of this paper. 
In Section \ref{f-sec2}, we collect some basic definitions and 
quickly recall Fujita's theory of $\Delta$-genera and 
the theory of quasi-log schemes. 
In Section \ref{f-sec3}, we explain three 
useful lemmas for quasi-log schemes for the 
reader's convenience. 
In Section \ref{f-sec4}, we give a detailed proof of Theorem \ref{f-thm1.6}. 
It is a combination of Fujita's theory of $\Delta$-genera and 
the theory of quasi-log schemes. 
In Section \ref{f-sec5}, we prove Theorem \ref{f-thm1.1}. 
Our proof is different from Koll\'ar's modified basepoint-freeness 
method in \cite{kollar-effective} and is new. It uses 
the framework of quasi-log schemes. 
In Section \ref{f-sec6}, we treat Theorem \ref{f-thm1.7}, which is 
a generalization of Theorem \ref{f-thm1.1}. The idea of the proof of 
Theorem \ref{f-thm1.7} is completely the same as that of the 
proof of Theorem \ref{f-thm1.1}. However, the proof of 
Theorem \ref{f-thm1.7} is harder than that of Theorem \ref{f-thm1.1}. 

\begin{ack} 
The author was partly supported by JSPS KAKENHI 
Grant Numbers JP16H03925, JP16H06337. 
He thanks Kento Fujita, Haidong Liu, and 
Keisuke Miyamoto very much for 
pointing out some mistakes in a preliminary 
version of this paper. 
He thanks Kento Fujita for 
informing him of \cite{andreatta-w} when 
he was writing \cite{fujino-miyamoto} with 
Keisuke Miyamoto. 
Finally, he would like to thank the referee for some useful comments and 
suggestions. 
\end{ack}

We will work over $\mathbb C$, the complex number field, throughout 
this paper. In this paper, a {\em{scheme}} means a separated 
scheme of finite type over $\mathbb C$. 
A {\em{variety}} means an integral scheme, that is, 
an integral separated scheme of finite type over $\mathbb C$.
We will use the theory of quasi-log schemes discussed 
in \cite[Chapter 6]{fujino-foundations}. 
\section{Preliminaries}\label{f-sec2} 

In this section, we collect some basic definitions of 
the minimal model program and the theory of quasi-log schemes. 
For the details, see \cite{fujino-fund} and 
\cite{fujino-foundations}. We also mention Fujita's 
$\Delta$-genera (see \cite{fujita1}), which will play a 
crucial role in this paper. 

\subsection{Basic definitions}\label{f-subsec2.1}

Let us recall singularities of pairs and some related 
definitions. 

\begin{defn}\label{f-def2.1} 
Let $X$ be a variety and let $E$ be a prime divisor 
on $Y$ for some birational morphism 
$f\colon Y\to X$ from a normal variety $Y$. 
Then $E$ is called a divisor {\em{over}} $X$. 
\end{defn}

\begin{defn}[Singularities of pairs]\label{f-def2.2}
A {\em{normal pair}} $(X, \Delta)$ consists of a normal 
variety $X$ and an $\mathbb R$-divisor 
$\Delta$ on $X$ such that $K_X+\Delta$ is $\mathbb R$-Cartier. 
Let $f\colon Y\to X$ be a projective 
birational morphism from a normal variety $Y$. 
Then we can write 
$$
K_Y=f^*(K_X+\Delta)+\sum _E a(E, X, \Delta)E
$$ 
with 
$$f_*\left(\underset{E}\sum a(E, X, \Delta)E\right)=-\Delta, 
$$ 
where $E$ runs over prime divisors on $Y$. 
We call $a(E, X, \Delta)$ the {\em{discrepancy}} of $E$ with 
respect to $(X, \Delta)$. 
Note that we can define the discrepancy $a(E, X, \Delta)$ for 
any prime divisor $E$ over $X$ by taking a suitable 
resolution of singularities of $X$. 
If $a(E, X, \Delta)\geq -1$ (resp.~$>-1$) for 
every prime divisor $E$ over $X$, 
then $(X, \Delta)$ is called {\em{sub log canonical}} (resp.~{\em{sub 
kawamata log terminal}}). 
We further assume that $\Delta$ is effective. 
Then $(X, \Delta)$ is 
called {\em{log canonical}} and {\em{kawamata log terminal}} 
if it is sub log canonical and sub kawamata log terminal, respectively. 
We simply say that $X$ has only {\em{kawamata log terminal singularities}} 
when $(X, 0)$ is a kawamata log terminal pair. 

Let $(X, \Delta)$ be a normal pair. 
If there exist a projective birational morphism 
$f\colon Y\to X$ from a normal variety $Y$ and a prime divisor $E$ on $Y$ 
such that $(X, \Delta)$ is 
sub log canonical in a neighborhood of the 
generic point of $f(E)$ and that 
$a(E, X, \Delta)=-1$, then $f(E)$ is called a {\em{log canonical center}} of 
$(X, \Delta)$. 
\end{defn}

\begin{defn}[Operations for $\mathbb R$-divisors]\label{f-def2.3}
Let $V$ be an equidimensional 
reduced scheme. 
An {\em{$\mathbb R$-divisor}} $D$ on $V$ is 
a finite formal sum 
$$
\sum _{i=1}^l d_i D_i, 
$$ 
where $D_i$ is an irreducible reduced closed subscheme 
of $V$ of pure codimension one with 
$D_i\ne D_j$ for $i\ne j$ and 
$d_i$ is a real number for every $i$. 
We put 
$$
D^{<1}=\sum _{d_i<1}d_iD_i, \quad 
D^{=1}=\sum _{d_i=1} D_i, \quad \text{and} 
\quad D^{>1}=\sum _{d_i>1}d_iD_i. 
$$ 
For every real number $x$, $\lceil x\rceil$ is the integer 
defined by $x\leq \lceil x\rceil <x+1$. 
Then we put 
$$
\lceil D\rceil =\sum _{i=1}^l \lceil d_i \rceil D_i 
\quad \text{and}\quad \lfloor D\rfloor =-\lceil -D\rceil. 
$$
\end{defn}

\begin{defn}[{Non-lc ideals and non-lc loci, see 
\cite{fujino-non-lc} and \cite[Section 7]{fujino-fund}}]\label{f-def2.4}
Let $(X, \Delta)$ be a normal pair such that 
$\Delta$ is effective and let 
$f\colon Y\to X$ be a resolution of singularities 
with 
$$
K_Y+\Delta_Y=f^*(K_X+\Delta)
$$ 
such that $\Supp\!\Delta_Y$ is a simple normal crossing 
divisor on $Y$. 
We put 
\begin{equation*}
\begin{split}
\mathcal J_{\NLC}(X, \Delta)&:=f_*\mathcal O_Y
(-\lfloor \Delta_Y\rfloor+\Delta^{=1}_Y)
\\ &=f_*\mathcal O_Y (\lceil -(\Delta^{<1}_Y)\rceil 
-\lfloor \Delta^{>1}_Y\rfloor)
\end{split}
\end{equation*}
and call it the {\em{non-lc ideal sheaf}} associated to 
the pair $(X, \Delta)$. 
We can check that $\mathcal J_{\NLC}(X, \Delta)$ is 
a well-defined ideal sheaf on $X$. 
The closed subscheme $\Nlc(X, \Delta)$ 
defined by $\mathcal J_{\NLC}(X, \Delta)$ is 
called the {\em{non-lc locus}} of $(X, \Delta)$. 
Note that $(X, \Delta)$ is log canonical 
if and only if $\mathcal J_{\NLC}(X, \Delta)=\mathcal 
O_X$. 
\end{defn}

\begin{defn}[$\sim _{\mathbb R}$ and $\equiv$]\label{f-def2.5}
Let $B_1$ and $B_2$ be $\mathbb R$-Cartier 
divisors on a scheme $X$. 
Then $B_1\sim _{\mathbb R} B_2$ means that 
$B_1$ is {\em{$\mathbb R$-linearly equivalent}} 
to $B_2$, that is, 
$B_1-B_2$ is a finite $\mathbb R$-linear combination 
of principal Cartier divisors. 
Let $f\colon X\to Y$ be a proper morphism between schemes. 
Then $B_1\equiv_YB_2$ means 
that $B_1$ is {\em{relatively numerically equivalent}} to 
$B_2$ {\em{over}} $Y$. 
When $Y$ is a point, 
we simply write $B_1\equiv B_2$ to denote 
$B_1\equiv _YB_2$ and say that 
$B_1$ is {\em{numerically equivalent}} to $B_2$. 
\end{defn}

\subsection{Fujita's $\Delta$-genera}\label{f-subsec2.2}
Let us quickly explain Fujita's theory of $\Delta$-genera, which 
will play a crucial role in this paper. We start with the definition of base 
loci. 

\begin{defn}[Base loci]\label{f-def2.6}
Let $f\colon X\to Y$ be a proper morphism between schemes and let 
$L$ be a Cartier divisor on $X$. 
Then $\Bs_f\!|L|$ denotes the support of 
$$\Coker\left(f^*f_*\mathcal O_X(L)\to \mathcal O_X(L)\right)$$ 
and 
is called the {\em{relative base locus}} of 
$|L|$. 
If $Y$ is a point, then 
we simply write $\Bs\!|L|$ to denote $\Bs_f\!|L|$. 
We can define $\Bs_f\!|\mathcal L|$ and $\Bs\!|\mathcal L|$ for 
every line bundle $\mathcal L$ on $X$ in the same way. 
\end{defn}

Let us recall the definition of Fujita's $\Delta$-genera. 
In this paper, we define $\Delta(V, L)$ only when $L$ is ample for simplicity. 
For the general case, see Fujita's original definition in \cite{fujita1}. 

\begin{defn}[{Fujita's $\Delta$-genera, see \cite[Definition 1.4]{fujita1}}]
\label{f-def2.7}
Let $V$ be a projective variety 
and let $L$ be an ample Cartier divisor on $V$. 
Then the {\em{$\Delta$-genus}} of $(V, L)$ is defined to be 
$$
\Delta(V, L)=\dim V +L^{\dim V}-\dim _{\mathbb C} H^0(V, 
\mathcal O_V(L)). 
$$ 
We can define $\Delta(V, \mathcal L)$ for every ample line bundle 
$\mathcal L$ in the same way. 
\end{defn}

The following famous theorem by Takao Fujita is one of the 
main ingredients of this paper. 
We recommend the interested reader to see 
Fujita's original statement (see \cite[Theorem 1.9]{fujita1}), 
which is more general than Theorem \ref{f-thm2.8}. 

\begin{thm}[{Fujita, see \cite[Theorem 1.9]{fujita1}}]\label{f-thm2.8}
Let $V$ be a projective variety and let 
$L$ be an ample Cartier divisor on $V$. 
Then the following inequality 
$$
\dim \Bs\!|L| <\Delta(V, L)
$$
holds, where $\dim \emptyset$ is defined to be $-\infty$. 
In particular, if $\Delta(V, L)=0$, 
then the complete linear system 
$|L|$ is basepoint-free. Of course, the same statement holds for 
ample line bundles $\mathcal L$. 
\end{thm}

\subsection{Quasi-log schemes}\label{f-subsec2.3}
The notion of quasi-log schemes was first introduced by 
Florin Ambro in order to 
establish the cone and contraction theorem for 
$(X, \Delta)$, where 
$X$ is a normal variety and $\Delta$ is an 
effective $\mathbb R$-divisor 
on $X$ such that $K_X+\Delta$ is $\mathbb R$-Cartier. 
Here we use the formulation in \cite[Chapter 6]
{fujino-foundations}, which is slightly different from 
Ambro's original one. 
We recommend the interested reader to see \cite[Appendix A]{fujino-pull} 
for the difference between our definition of quasi-log schemes 
and Ambro's one. 

\medskip 

In order to define quasi-log schemes, we need the notion of 
globally embedded simple normal crossing pairs. 

\begin{defn}[{Globally embedded simple normal crossing 
pairs, see \cite[Definition 6.2.1]{fujino-foundations}}]\label{f-def2.9} 
Let $Y$ be a simple normal crossing divisor 
on a smooth 
variety $M$ and let $D$ be an $\mathbb R$-divisor 
on $M$ such that 
$\Supp (D+Y)$ is a simple normal crossing divisor on $M$ and that 
$D$ and $Y$ have no common irreducible components. 
We put $B_Y=D|_Y$ and consider the pair $(Y, B_Y)$. 
We call $(Y, B_Y)$ a {\em{globally embedded simple normal 
crossing pair}} and $M$ the {\em{ambient space}} 
of $(Y, B_Y)$. A {\em{stratum}} of $(Y, B_Y)$ is a log canonical  
center of $(M, Y+D)$ that is contained in $Y$. 
\end{defn}

Let us recall the definition of quasi-log schemes. 

\begin{defn}[{Quasi-log 
schemes, see \cite[Definition 6.2.2]{fujino-foundations}}]\label{f-def2.10}
A {\em{quasi-log scheme}} is a scheme $X$ endowed with an 
$\mathbb R$-Cartier divisor 
(or $\mathbb R$-line bundle) 
$\omega$ on $X$, a closed subscheme 
$X_{-\infty}\subsetneq X$, and a finite collection $\{C\}$ of reduced 
and irreducible subschemes of $X$ such that there is a 
proper morphism $f\colon (Y, B_Y)\to X$ from a globally 
embedded simple 
normal crossing pair satisfying the following properties: 
\begin{itemize}
\item[(1)] $f^*\omega\sim_{\mathbb R}K_Y+B_Y$. 
\item[(2)] The natural map 
$\mathcal O_X
\to f_*\mathcal O_Y(\lceil -(B_Y^{<1})\rceil)$ 
induces an isomorphism 
$$
\mathcal I_{X_{-\infty}}\overset{\simeq}{\longrightarrow} 
f_*\mathcal O_Y(\lceil 
-(B_Y^{<1})\rceil-\lfloor B_Y^{>1}\rfloor),  
$$ 
where $\mathcal I_{X_{-\infty}}$ is the defining ideal sheaf of 
$X_{-\infty}$. 
\item[(3)] The collection of reduced and irreducible subschemes 
$\{C\}$ coincides with the images 
of the strata of $(Y, B_Y)$ that are not included in $X_{-\infty}$. 
\end{itemize}
We simply write $[X, \omega]$ to denote 
the above data 
$$
\bigl(X, \omega, f\colon (Y, B_Y)\to X\bigr)
$$ 
if there is no risk of confusion. 
The reduced and irreducible subschemes $C$ 
are called the {\em{qlc strata}} of $[X, \omega]$, 
$X_{-\infty}$ is called the {\em{non-qlc locus}} 
of $[X, \omega]$, and $f\colon (Y, B_Y)\to X$ is 
called a {\em{quasi-log resolution}} 
of $[X, \omega]$. 
We sometimes use $\Nqlc(X, 
\omega)$ to denote 
$X_{-\infty}$. 
If a qlc stratum $C$ of $[X, \omega]$ is not an 
irreducible component of $X$, then 
it is called a {\em{qlc center}} of $[X, \omega]$. 
\end{defn}

\begin{rem}\label{f-rem2.11} 
By restricting the isomorphism 
$$
\mathcal I_{X_{-\infty}}\overset{\simeq}{\longrightarrow} 
f_*\mathcal O_Y(\lceil 
-(B_Y^{<1})\rceil-\lfloor B_Y^{>1}\rfloor)
$$ in Definition \ref{f-def2.10} 
to the Zariski open set $U=X\setminus X_{-\infty}$, we have 
$$
\mathcal O_U\overset{\simeq}{\longrightarrow} 
f_*\mathcal O_{f^{-1}(U)}(\lceil 
-(B_Y^{<1})\rceil). 
$$ 
This implies that 
$$
\mathcal O_U\overset{\simeq}{\longrightarrow} 
f_*\mathcal O_{f^{-1}(U)} 
$$ holds 
since $\lceil 
-(B_Y^{<1})\rceil$ is effective. 
Hence, $f\colon f^{-1}(U)\to U$ is surjective and 
has connected fibers. 
Note that a qlc stratum $C$ of $[X, \omega]$ is 
the image of some stratum of 
$(Y, B_Y)$ that is not included in $X_{-\infty}$. 
Therefore, $X$ is the union of $\{C\}$ and $X_{-\infty}$. 
In particular, any irreducible component of $X$ that is 
not included in $X_{-\infty}$ is a qlc stratum of 
$[X, \omega]$. 
\end{rem}

\begin{defn}[{Quasi-log canonical 
pairs, see \cite[Definition 6.2.9]{fujino-foundations}}]
\label{f-def2.12}
Let 
$$
\left(X, \omega, f\colon (Y, B_Y)\to X\right)
$$ 
be a quasi-log scheme. 
If $X_{-\infty}=\emptyset$, then 
it is called a {\em{quasi-log canonical pair}}. 
\end{defn}

The most important result in the theory of quasi-log scheme is 
adjunction and the following vanishing theorem. 
We will repeatedly use Theorem \ref{f-thm2.13} in this paper. 
The proof of Theorem \ref{f-thm2.13} in \cite{fujino-foundations} 
heavily depends on the theory of mixed Hodge structures on 
cohomology with compact support (see \cite[Chapter 5]{fujino-foundations}). 

\begin{thm}[{see \cite[Theorem 6.3.5]{fujino-foundations}}]\label{f-thm2.13}
Let $[X, \omega]$ be a quasi-log scheme and let $X'$ be 
the union of $X_{-\infty}$ with a {\em{(}}possibly empty{\em{)}} 
union of some 
qlc strata of $[X, \omega]$. Then we have the following 
properties.
\begin{itemize}
\item[(i)] {\em{(Adjunction).}} Assume that $X'\ne X_{-\infty}$. 
Then $X'$ is a quasi-log scheme with $\omega'=\omega|_{X'}$ 
and $X'_{-\infty}=X_{-\infty}$. 
Moreover, the qlc strata of $[X', \omega']$ 
are exactly the qlc strata of $[X, \omega]$ that are included in $X'$.
\item[(ii)] {\em{(Vanishing theorem).}} Assume 
that $\pi\colon X\to S$ is a proper morphism 
between schemes. Let $L$ be a Cartier divisor on $X$ 
such that $L-\omega$ is ample over $S$ with respect to 
$[X, \omega]$. Then $R^i\pi_*(\mathcal I_{X'}\otimes 
\mathcal O_X(L))=0$ for every 
$i>0$, where $\mathcal I_{X'}$ 
is the defining ideal sheaf of $X'$ on $X$.
\end{itemize}
\end{thm}

We quickly explain the main idea of the proof of Theorem \ref{f-thm2.13} (i) 
for the reader's convenience. 
For the details, see \cite[Theorem 6.3.5]{fujino-foundations}. 

\begin{proof}[Idea of Proof of Theorem \ref{f-thm2.13} (i)] 
By definition, $X'$ is the union of $X_{-\infty}$ with 
a union of some qlc strata of $[X, \omega]$ set theoretically. 
We assume that $X'\ne X_{-\infty}$ holds. 
By \cite[Proposition 6.3.1]{fujino-foundations}, 
we may assume that 
the union of all strata of $(Y, B_Y)$ mapped to $X'$ by $f$, which 
is denoted by $Y'$, is a union of some irreducible components of $Y$. 
We put $Y''=Y-Y'$, 
$K_{Y''}+B_{Y''}=(K_Y+B_Y)|_{Y''}$, and 
$K_{Y'}+B_{Y'}=(K_Y+B_Y)|_{Y'}$. We set $f''=f|_{Y''}$ and 
$f'=f|_{Y'}$. Then we claim that 
$$
\left(X', \omega', f'\colon (Y', B_{Y'})\to X'\right)
$$ 
becomes a quasi-log scheme satisfying the desired properties. 
Let us consider the 
following short exact sequence:   
\begin{equation*}
\begin{split}
0&\to \mathcal O_{Y''}(\lceil -(B^{<1}_{Y''})\rceil -\lfloor 
B^{>1}_{Y''}\rfloor-Y'|_{Y''})\to 
\mathcal O_Y(\lceil -(B^{<1}_Y)\rceil -\lfloor 
B^{>1}_Y\rfloor)
\\ &\to 
\mathcal O_{Y'}(\lceil -(B^{<1}_{Y'})\rceil -\lfloor 
B^{>1}_{Y'}\rfloor)\to 0, 
\end{split} 
\end{equation*} 
which is induced by 
$$
0\to \mathcal O_{Y''}(-Y'|_{Y''})\to \mathcal O_Y\to 
\mathcal O_{Y'}\to 0. 
$$
We take the associated long exact sequence. 
Then we can check that the connecting homomorphism 
$$
\delta\colon  f'_*\mathcal O_{Y'}(\lceil -(B^{<1}_{Y'})\rceil -\lfloor 
B^{>1}_{Y'}\rfloor)\to 
R^1f''_* \mathcal O_{Y''}(\lceil -(B^{<1}_{Y''})\rceil -\lfloor 
B^{>1}_{Y''}\rfloor-Y'|_{Y''})
$$ 
is zero by using a generalization of Koll\'ar's torsion-freeness based on the 
theory of mixed Hodge structures on cohomology with compact support 
(see \cite[Chapter 5]{fujino-foundations}). 
We put 
$$
\mathcal I_{X'}:=f''_* \mathcal O_{Y''}(\lceil -(B^{<1}_{Y''})\rceil -\lfloor 
B^{>1}_{Y''}\rfloor-Y'|_{Y''}), 
$$ which is an ideal sheaf on $X$ since 
$\mathcal I_{X'}\subset \mathcal I_{X_{-\infty}}$, 
and define a scheme structure on $X'$ by $\mathcal I_{X'}$. 
Then we obtain the following big commutative diagram:  
$$
\xymatrix{
& 0 \ar[d]& 0\ar[d] & &\\ 
0\ar[r]& f''_* \mathcal O_{Y''}(\lceil -(B^{<1}_{Y''})\rceil -\lfloor 
B^{>1}_{Y''}\rfloor-Y'|_{Y''})\ar[r]^-{=}\ar[d] &\mathcal I_{X'}\ar[d]&&
\\ 
0\ar[r]&f_*\mathcal O_Y(\lceil -(B^{<1}_Y)\rceil -\lfloor 
B^{>1}_Y\rfloor)=\mathcal I_{X_{-\infty}}\ar[r]\ar[d]
&\mathcal O_X\ar[d]\ar[r]&\mathcal O_{X_{-\infty}}\ar@{=}[d]
\ar[r]& 0\\ 
0\ar[r]&f'_* \mathcal O_{Y'}(\lceil -(B^{<1}_{Y'})\rceil -\lfloor 
B^{>1}_{Y'}\rfloor)=\mathcal I_{X'_{-\infty}}
\ar[d]\ar[r]&\mathcal O_{X'}\ar[d]\ar[r]&\mathcal O_{X'_{-\infty}}\ar[r]&0\\ 
&0&0&&
}
$$ 
by the above arguments. More precisely, 
by the above big commutative diagram, 
$$
\mathcal I_{X'_{-\infty}}=f'_* \mathcal O_{Y'}(\lceil -(B^{<1}_{Y'})\rceil -\lfloor 
B^{>1}_{Y'}\rfloor)
$$ 
is an ideal sheaf on $X'$ such that $\mathcal O_X/\mathcal I_{X_{-\infty}}
=\mathcal O_{X'}/\mathcal I_{X'_{-\infty}}$. 
Thus we obtain that 
$$
\left(X', \omega', f'\colon (Y', B_{Y'})\to X'\right)
$$ 
is a quasi-log scheme satisfying the desired properties. 
\end{proof}

The following example is very important. It 
shows that we can treat log canonical 
pairs as quasi-log canonical pairs. 

\begin{ex}[{\cite[6.4.1]{fujino-foundations}}]\label{f-ex2.14}
Let $(X, \Delta)$ be a normal pair such that 
$\Delta$ is effective. 
Let $f\colon Y\to X$ be a resolution of singularities such that 
$$
K_Y+B_Y=f^*(K_X+\Delta)
$$ 
and that $\Supp\!B_Y$ is a simple normal crossing divisor on $Y$. 
We put $\omega=K_X+\Delta$. 
Then $K_Y+B_Y\sim _{\mathbb R}f^*\omega$ holds. 
Since $\Delta$ is effective, $\lceil -(B^{<1}_Y)\rceil$ is effective 
and $f$-exceptional. 
Therefore, 
the natural map 
$$
\mathcal O_X\to f_*\mathcal O_Y(\lceil -(B^{<1}_Y)\rceil)
$$ 
is an isomorphism. 
We put 
$$
\mathcal I_{X_{-\infty}}:=\mathcal J_{\NLC}(X, \Delta)
=f_*\mathcal O_Y(\lceil -(B^{<1}_Y)\rceil-\lfloor B^{>1}_Y\rfloor),  
$$ 
where $\mathcal J_{\NLC}(X, \Delta)$ is the non-lc 
ideal sheaf associated to $(X, \Delta)$ in Definition \ref{f-def2.4}. 
We put $M=Y\times \mathbb C$ and $D=B_Y\times 
\mathbb C$. 
Then $(Y, B_Y)\simeq (Y\times \{0\}, B_Y\times 
\{0\})$ is a globally embedded simple normal crossing pair. 
Thus 
$$
\left(X, \omega, f\colon(Y, B_Y)\to X\right)
$$
becomes a quasi-log scheme. 
By construction, $(X, \Delta)$ is log canonical 
if and only if $[X, \omega]$ is quasi-log canonical. 
We note that $C$ is a log canonical 
center of $(X, \Delta)$ if and only if $C$ is a qlc center 
of $[X, \omega]$. We also note that 
$X$ itself is a qlc stratum of $[X, \omega]$. 

Let $X'$ be the union of $X_{-\infty}$ with 
a union of some qlc centers of $[X, \omega]$. 
If $X'\ne X_{-\infty}$, then $[X', \omega|_{X'}]$ naturally 
becomes a quasi-log scheme by adjunction 
(see Theorem \ref{f-thm2.13} (i) and \cite[Theorem 6.3.5 (i)]{fujino-foundations}). 
When $X_{-\infty}=\emptyset$, equivalently, $(X, \Delta)$ 
is log canonical, we see that 
$[X', \omega|_{X'}]$ is quasi-log canonical. 
By construction, $X'$ is not necessarily equidimensional and 
is a highly singular 
reducible and reduced scheme. 
\end{ex}

For the basic properties of quasi-log schemes, 
see \cite[Chapter 6]{fujino-foundations}. 

\section{Three lemmas for quasi-log schemes}
\label{f-sec3}

In this section, 
we will explain three useful lemmas for quasi-log schemes for the 
reader's convenience. 
They are essentially contained in \cite[Chapter 6]{fujino-foundations} 
or easily follow from the arguments in \cite[Chapter 6]{fujino-foundations}. 

\medskip 

Let us start with the following easy lemma, which is 
almost obvious by definition. 

\begin{lem}\label{f-lem3.1}
Let $$\left(X, \omega, f\colon  (Y, B_Y)\to X\right)$$ be 
a quasi-log canonical 
pair and let 
$B$ be an effective 
$\mathbb R$-Cartier divisor on $X$. 
Assume that 
$(Y, B_Y+f^*B)$ is a globally embedded simple normal 
crossing pair. 
Then 
$$
\left(X, \omega+B, f\colon  (Y, B_Y+f^*B)\to X\right)
$$ 
is a quasi-log scheme. 
Of course, $[X, \omega+B]$ is quasi-log canonical 
if and only if $B_Y+f^*B$ is a subboundary 
$\mathbb R$-divisor on $Y$, that is, 
$(B_Y+f^*B)^{>1}=0$. 
\end{lem}
\begin{proof}
By definition, $K_Y+B_Y\sim _{\mathbb R} f^*\omega$. 
Therefore, $K_Y+B_Y+f^*B\sim 
_{\mathbb R} f^*(\omega+B)$ obviously holds true. 
Since $[X, \omega]$ is a quasi-log canonical 
pair, the natural map 
$$
\mathcal O_X\to 
f_*\mathcal O_Y(\lceil -(B^{<1}_Y)\rceil)
$$ 
is an isomorphism. 
Since it factors through 
$f_*\mathcal O_Y$, we have 
\begin{equation}\label{f-eq3.1} 
\mathcal O_X\overset{\simeq}{\longrightarrow} 
f_*\mathcal O_Y\overset{\simeq}{\longrightarrow} 
f_*\mathcal O_Y(\lceil -(B^{<1}_Y)\rceil). 
\end{equation}
We note that 
$$
0\leq \lceil -(B_Y+f^*B)^{<1}\rceil 
\leq \lceil -(B^{<1}_Y)\rceil. 
$$ 
Therefore, we obtain 
$$
\mathcal O_X\overset{\simeq}{\longrightarrow} 
f_*\mathcal O_Y\overset{\simeq}{\longrightarrow}  
f_*\mathcal O_Y(\lceil -(B_Y+f^*B)^{<1}\rceil)
\overset{\simeq}{\longrightarrow} 
f_*\mathcal O_Y(\lceil -(B^{<1}_Y)\rceil). 
$$ 
Thus, we get a nonzero coherent ideal sheaf 
$$
\mathcal I_{\Nqlc(X, \omega+B)}:=
f_*\mathcal O_Y(\lceil -(B_Y+f^*B)^{<1}\rceil
-\lfloor (B_Y+f^*B)^{>1}\rfloor), 
$$ 
which defines a closed subscheme $\Nqlc(X, \omega+B)$. 
Let $W$ be a reduced and irreducible 
subscheme of $X$. 
We say that $W$ is a qlc stratum of $[X, \omega+B]$ if 
$W$ is not included in $\Nqlc(X, \omega+B)$ 
and is the $f$-image of some stratum of $(Y, B_Y+f^*B)$. 
Then 
$$
\left(X, \omega+B, f\colon  (Y, B_Y+f^*B)\to X\right)
$$ 
is a quasi-log scheme. 
By construction, $[X, \omega+B]$ is a 
quasi-log canonical pair if and only if $(B_Y+f^*B)^{>1}=0$. 
Note that $\left(X, \omega+B, f\colon (Y, B_Y+f^*B)\to X\right)$ coincides 
with 
$\left(X, \omega, f\colon  (Y, B_Y)\to X\right)$ outside $\Supp\!B$.  
\end{proof}

The next lemma is similar to the previous one. 
However, the proof is not so obvious because we need 
the argument in the proof of adjunction (see Theorem \ref{f-thm2.13} (i)). 

\begin{lem}\label{f-lem3.2}
Let $$\left(X, \omega, f\colon  (Y, B_Y)\to X\right)$$ be 
a quasi-log scheme and let 
$B$ be an effective 
$\mathbb R$-Cartier divisor on $X$. 
Let $X'$ be the union of $\Nqlc(X, \omega)$ and all qlc centers of $[X, \omega]$ 
contained in $\Supp\!B$. 
Assume that 
the union of all strata of $(Y, B_Y)$ mapped to $X'$ by $f$, which is 
denoted by $Y'$, is a union of some irreducible components 
of $Y$. 
We put $Y''=Y-Y'$, $K_{Y''}+B_{Y''}=
(K_Y+B_Y)|_{Y''}$, and $f''=f|_{Y''}$. 
We further assume that 
$$
\left(Y'', B_{Y''}+(f'')^*B\right)
$$ 
is a globally embedded simple normal crossing pair. 
Then 
$$
\left(X, \omega+B, f''\colon  (Y'', B_{Y''}+(f'')^*B)\to X\right)
$$ 
is a quasi-log scheme. 
\end{lem}
\begin{proof}
Since $K_Y+B_Y\sim _{\mathbb R} f^*\omega$, we have 
$K_{Y''}+B_{Y''}\sim _{\mathbb R} (f'')^*\omega$. 
Therefore, $K_{Y''}+B_{Y''}+(f'')^*B\sim 
_{\mathbb R}(f'')^*(\omega+B)$ holds true. 
By the proof of adjunction (see 
Theorem \ref{f-thm2.13} (i) and 
\cite[Theorem 6.3.5 (i)]{fujino-foundations}), 
we have 
$$
\mathcal I_{X'}=f''_*\mathcal O_{X''}(\lceil -(B^{<1}_{Y''})\rceil 
-\lfloor B^{>1}_{Y''}\rfloor-Y'|_{Y''}), 
$$ 
where $\mathcal I_{X'}$ is the defining ideal sheaf of $X'$ on $X$. 
Note that the following key inequality 
$$
\lceil -(B_{Y''}+(f'')^*B)^{<1}\rceil -\lfloor 
(B_{Y''}+(f'')^*B)^{>1}\rfloor 
\leq \lceil -(B^{<1}_{Y''})\rceil -\lfloor 
B^{>1}_{Y''}\rfloor -Y'|_{Y''}
$$
holds. 
Therefore, we put 
$$
\mathcal I_{\Nqlc(X, \omega+B)}:= 
f''_*\mathcal O_{Y''}(\lceil -(B_{Y''}+(f'')^*B)^{<1}\rceil -\lfloor 
(B_{Y''}+(f'')^*B)^{>1}\rfloor) \subset \mathcal I_{X'}\subset 
\mathcal O_X
$$ 
and define a closed subscheme $\Nqlc(X, \omega+B)$ of $X$ 
by $\mathcal I_{\Nqlc(X, \omega+B)}$. 
Then 
$$
\left(X, \omega+B, f''\colon  (Y'', B_{Y''}+(f'')^*B)\to X\right)
$$ 
is a quasi-log scheme. 
Let $W$ be a reduced and irreducible subscheme of $X$. 
As usual, we say that $W$ is a qlc stratum of $[X, \omega+B]$ 
when $W$ is not contained in $\Nqlc(X, \omega+B)$ and is 
the $f''$-image of some stratum of $(Y'', B_{Y''}+(f'')^*B)$. 
By construction, we have $X'\subset \Nqlc(X, \omega+B)$. 
We note that 
$\left(X, \omega+B, f''\colon  (Y'', B_{Y''}+(f'')^*B)\to X\right)$ coincides 
with $\left(X, \omega, f\colon (Y, B_Y)\to X\right)$ outside $\Supp\!B$. 
\end{proof}

The final lemma is easy but very useful. 
We often use it without mentioning 
it explicitly. 

\begin{lem}[Bertini-type theorem]\label{f-lem3.3}
Let $[X, \omega]$ be a quasi-log scheme and 
let $\Lambda$ be a free linear system on $X$. 
If $D$ is a general member of $\Lambda$, then 
$[X, \omega+cD]$ becomes a quasi-log scheme with 
$\Nqlc(X, \omega+cD)=\Nqlc(X, \omega)$ for 
every $0\leq c\leq 1$. 

More precisely, there exists a proper morphism 
$f\colon (Y, B_Y)\to X$ from a globally embedded simple normal 
crossing pair $(Y, B_Y)$ such that 
$(Y, B_Y+f^*D)$ is a globally embedded simple normal crossing 
pair and that 
$$
\left( X, \omega+cD, f\colon  (Y, B_Y+f^*cD)\to X\right)
$$ 
is a quasi-log scheme with $\Nqlc(X, \omega+cD)=\Nqlc(X, \omega)$ for 
every $0\leq c\leq 1$. 

When $c=1$, every irreducible component $D^\dag$ of $D$ is a qlc 
center of $$\left(X, \omega+D, f\colon (Y, B_Y+f^*D)\to X\right). 
$$ 
Therefore, by adjunction, 
$[D', (\omega+D)|_{D'}]$ is a quasi-log scheme, where 
$D'=D^\dag\cup \Nqlc(X, \omega)$. 
\end{lem}

\begin{proof}
Let $f\colon (Y, B_Y)\to X$ be a quasi-log resolution 
of $[X, \omega]$. Let $\nu\colon Y^\nu\to Y$ be the normalization of $Y$ 
with $K_{Y^\nu}
+\Theta=\nu^*(K_Y+B_Y)$ as usual. 
If $D$ is a general member of $\Lambda$, then 
$\nu^*f^*D$ is smooth, $\nu^*f^*D$ and $\Theta$ have no common 
components, and $\Supp (\nu^*f^*D+\Theta)$ is a simple 
normal crossing divisor on $Y^\nu$. 
By taking some blow-ups along irreducible components 
of $f^*D$ repeatedly (see \cite[Lemma 5.8.8]{fujino-foundations}), 
we may further assume that 
$(Y, B_Y+f^*D)$ is a globally embedded simple 
normal crossing pair (see \cite[Proposition 6.3.1]{fujino-foundations}). 
Since 
$$
\lfloor (B_Y+f^*cD)^{>1}\rfloor =\lfloor B^{>1}_Y\rfloor 
\quad \text{and} 
\quad 
0\leq \lceil -(B_Y+f^*cD)^{<1}\rceil = 
\lceil -(B^{<1}_Y)\rceil
$$ 
hold for every $0\leq c\leq 1$, 
we obtain that the following equality 
$$
f_*\mathcal O_Y(\lceil -(B_Y+f^*cD)^{<1}\rceil-
\lfloor (B_Y+f^*cD)^{>1}\rfloor)=f_*\mathcal O_Y(\lceil -(B^{<1}_Y)\rceil
-\lfloor B^{>1}_Y\rfloor). 
$$ 
holds true for every $0\leq c\leq 1$. 
Therefore, we obtain that 
$$
\left(X, \omega+cD, f\colon  (Y, B_Y+f^*cD)\to X\right)
$$ 
is a quasi-log scheme with 
$\Nqlc(X, \omega+cD)=\Nqlc(X, \omega)$ for 
every $0\leq c\leq 1$. 
By construction, the quasi-log scheme structure of $[X, \omega+cD]$ is 
independent of $c$ outside $\Supp\!D$. 
It is obvious that every irreducible component $D^\dag$ of $D$ is a qlc 
center of $[X, \omega+D]$. 
Therefore, by adjunction (see Theorem \ref{f-thm2.13} (i)), 
we obtain the desired statement. 
\end{proof}

In order to explain how to 
make new quasi-log scheme structures, let us 
treat the following proposition. 

\begin{prop}\label{f-prop3.4}
Let $[X, \omega]$ be a quasi-log scheme and 
let $L$ be a Cartier divisor on $X$ such that 
$\Bs\!|L|$ contains no qlc strata of $[X, \omega]$ and 
that $\Bs\!|L|$ is disjoint from $X_{-\infty}$. 
If $D$ is a general member of $|L|$. 
Then there exists $0< c\leq 1$ such that 
$[X, \omega+cD]$ becomes a quasi-log scheme with 
$\Nqlc(X, \omega+cD)=\Nqlc(X, \omega)$ and 
that there exists a qlc center $C$ of $[X, \omega+cD]$ 
with $C\cap \Bs\!|L|\ne \emptyset$. 
\end{prop}

\begin{proof}
Let $f\colon (Y, B_Y)\to X$ be a quasi-log resolution of $[X, \omega]$. 
Since $D$ is a general member of $|L|$, $\Bs\!|L|$ contains 
no qlc strata of $[X, \omega]$, and $\Bs\!|L|\cap 
X_{-\infty}=\emptyset$, $f^*D$ is a well-defined Cartier 
divisor on $Y$. We note that 
$[X, \omega+cD]$ becomes a quasi-log scheme 
with $\Nqlc(X, \omega+cD)=\Nqlc(X, \omega)$ outside 
$\Bs\!|L|$ for every $0\leq c\leq 1$ by Lemma \ref{f-lem3.3}. 

By taking a suitable birational modification of the 
ambient space $M$ of $(Y, B_Y)$ 
(see \cite[Proposition 6.3.1]{fujino-foundations}), 
we may assume that 
$$
(Y, f^*D+\Supp\!B_Y)
$$ 
is a globally embedded simple normal crossing pair. 
We may further assume that 
$f^*D$ and $\Supp\!B_Y$ have no common components 
outside $f^{-1}\Bs\!|L|$ and that 
$f^*D$ is reduced outside $f^{-1}\Bs\!|L|$. 

We put 
$$
c=\sup \{t\in \mathbb R \, |\,  \text{$(tf^*D+B_Y)^{>1}=0$ 
holds over $X\setminus X_{-\infty}$}\}. 
$$ 
Then we have: 
\begin{claim}
We have $0< c\leq 1$. 
\end{claim}
\begin{proof}[Proof of Claim]
By replacing $X$ with $X\setminus X_{-\infty}$, we may 
assume that $X_{-\infty}=\emptyset$.  
Therefore, the natural map 
$$
\mathcal O_X\to f_*\mathcal O_Y(\lceil -(B^{<1}_Y)\rceil)
$$ 
is an isomorphism. 
Since $B^{>1}_Y=0$ by $X_{-\infty}=\emptyset$, 
the inequality $0<c$ is obvious because 
$D$ is a general member of $|L|$ and 
$\Bs\!|L|$ contains no qlc strata of $[X, \omega]$. 
We assume that the inequality $c>1$ holds. 
Then the natural map 
$$
\mathcal O_X\to f_*\mathcal O_Y(\lceil -(B^{<1}_Y)\rceil)
$$ factors through $\mathcal O_X(D)$, that is, we have:  
$$
\mathcal O_X\hookrightarrow 
\mathcal O_X(D)\to f_*\mathcal O_Y(\lceil -(B^{<1}_Y)\rceil). 
$$ 
This is a contradiction. 
Hence we get the desired inequality $c\leq 1$. 
\end{proof}
We consider 
$$
\left(X, \omega+cD, f\colon (Y, B_Y+cf^*D)\to 
X\right). 
$$ 
It is obvious that $f^*(\omega+cD)
\sim _{\mathbb R} K_Y+B_Y+cf^*D$ holds 
since $f^*\omega\sim _{\mathbb R} K_Y+B_Y$. 
We note that 
$$
0\leq \lceil -(B_Y+cf^*D)^{<1}\rceil 
\leq \lceil -(B^{<1}_Y)\rceil 
$$ 
obviously holds and 
that 
$$
\lceil -(B_Y+cf^*D)^{<1}\rceil -\lfloor (B_Y+cf^*D)^{>1}\rfloor 
=\lceil -(B^{<1}_Y)\rceil -\lfloor B^{>1}_Y\rfloor  
$$ holds 
over a neighborhood of $X_{-\infty}$. 
Therefore, 
$$
\left(X, \omega+cD, f\colon (Y, B_Y+cf^*D)\to 
X\right). 
$$ is a quasi-log scheme with 
$\Nqlc(X, \omega+cD)=
\Nqlc(X, \omega)$. 

If $c=1$, then we see that 
every irreducible 
component $D^\dag$ of $\Supp\!D$ with $D^\dag
\not\subset X_{-\infty}$ is a qlc center of $[X, \omega+D]$ 
by the proof of Claim. 
Therefore, we can find a qlc center 
$C$ of $[X, \omega+D]$ with $C\cap 
\Bs\!|L|\ne \emptyset$. 

If $c<1$, then we can find an irreducible component 
$G$ of $(cf^*D+B_Y)^{=1}$ such that 
$f(G)\cap \Bs\!|L|\ne \emptyset$ by 
construction. 
Thus $C:=f(G)$ is a desired qlc center of $[X, \omega+cD]$. 
\end{proof}

\section{Proof of Theorem \ref{f-thm1.6}}
\label{f-sec4}

In this section, we will prove Theorem \ref{f-thm1.6}, which may look 
artificial but is very useful. 

\medskip

Let us start with an easy lemma, which 
follows from Fujita's theory of $\Delta$-genera 
(see \cite{fujita1}). 

\begin{lem}\label{f-lem4.1}
Let $[X, \omega]$ be a projective quasi-log canonical 
pair such that 
$X$ is irreducible with $n=\dim X\geq 1$. 
Let $L$ be an ample Cartier divisor on $X$ such that 
$\omega+rL\equiv 0$. 
Then the inequality $r\leq n+1$ holds. 
We further assume that $r>n-1$ holds. 
Then the complete linear system $|L|$ is basepoint-free. 
\end{lem}

\begin{proof}
Let us consider 
$$
\chi(t):=\chi (X, \mathcal O_X(tL))=\sum _{i=0}^n (-1)^i \dim _{\mathbb C} 
H^i(X, \mathcal O_X(tL)). 
$$
Since $L$ is ample, $\chi(t)$ is a nontrivial polynomial with 
$\deg\chi (t)=\dim X=n$. 
\begin{step}\label{f-4.1-step1}
In this step, we will prove that $r\leq n+1$. 

We assume that $r>n+1$ holds. Then 
$$
H^i(X, \mathcal O_X(tL))=0
$$ 
for $i>0$ and $t\in \mathbb Z$ with 
$t\geq -(n+1)$ since 
$tL-\omega\equiv (t+r)L$ is ample for 
$t\geq -(n+1)$ (see Theorem \ref{f-thm2.13} (ii)). 
On the other hand, 
$$
H^0(X, \mathcal O_X(tL))=0
$$ 
for $t<0$ since $L$ is ample. 
Therefore, we have $\chi(t)=0$ for $t=-1, \ldots, -(n+1)$. 
This implies that $\chi(t)\equiv 0$ holds. 
This is a contradiction. Hence we obtain the desired 
inequality $r\leq n+1$. 
\end{step}
\begin{step}\label{f-4.1-step2}
In this step, we will prove that $|L|$ is basepoint-free under 
the assumption that $r>n-1$ holds. 

As in Step \ref{f-4.1-step1}, 
we have $\chi(t)=0$ for $t=-1, \ldots, -(n-1)$ since 
$r>n-1$ by assumption. 
Therefore, we get 
$$
\chi (X, \mathcal O_X(tL))=\frac{1}{n!} 
(\alpha t+\beta)(t+1)\cdots (t+n-1)
$$ 
for some rational numbers $\alpha$ and $\beta$. 
It is well known that $\alpha=L^n$. 
We note that 
$$
\chi (X, \mathcal O_X)=\dim _\mathbb C H^0(X, \mathcal O_X)=1. 
$$
Therefore, $\beta=n$ holds. 
Hence we obtain 
$$
\dim _{\mathbb C} H^0(X, \mathcal O_X(L))=L^n+n. 
$$
This implies that 
$$
\Delta(X, L)=L^n+n-\dim _{\mathbb C} H^0(X, \mathcal O_X(L))=0
$$ 
holds. 
Thus we obtain that $|L|$ is basepoint-free by 
Theorem \ref{f-thm2.8} (see also  
\cite[Corollary 1.10]{fujita1}). 
\end{step}
We obtained all the desired statements. 
\end{proof}

The following example shows that 
the assumption $r>n-1$ in Lemma \ref{f-lem4.1} is 
sharp. 

\begin{ex}\label{f-ex4.2} 
Let $X$ be a Del Pezzo surface of degree one. 
We put $L=-K_X$. 
Then $K_X+rL=0$ with $r=1$ holds. 
We note that $r=1=2-1=\dim X-1$ holds. 
It is easy to check that 
$|L|=|-K_X|$ is not basepoint-free. 
\end{ex}

We can prove the following corollary. 

\begin{cor}\label{f-cor4.3}
Let $[X, \omega]$ be a projective quasi-log canonical 
pair. Note that $X$ may be reducible. 
Let $L$ be an ample Cartier divisor on $X$ such that $\omega+rL\equiv 0$ with 
$r>n-1$, where $n=\dim X$. 
Then the complete linear system $|L|$ is basepoint-free. 
\end{cor}

\begin{proof}
Let $X_i$ be any irreducible component of 
$X$. Since $X_i$ is a qlc stratum of $[X, \omega]$, 
$[X_i, \omega|_{X_i}]$ is a quasi-log canonical 
pair by adjunction (see Theorem \ref{f-thm2.13} (i)). 
If $\dim X_i=0$, then $|L|_{X_i}|$ 
is obviously basepoint-free. When 
$\dim X_i>0$, the complete linear system 
$|L|_{X_i}|$ is basepoint-free by 
Lemma \ref{f-lem4.1} because 
$\omega|_{X_i} +rL|_{X_i}\equiv 0$ with 
$r>\dim X_i -1$. 
Since $L-\omega\equiv (r+1)L$ is ample, 
we have $H^1(X, \mathcal I_{X_i}\otimes 
\mathcal O_X(L))=0$ by 
Theorem \ref{f-thm2.13} (ii), 
where $\mathcal I_{X_i}$ is the defining ideal sheaf of 
$X_i$ on $X$. 
Therefore the restriction map 
$$
H^0(X, \mathcal O_X(L))\to H^0(X_i, \mathcal O_{X_i}
(L))
$$ 
is surjective. 
This implies that $|L|$ is basepoint-free. 
\end{proof}

Let us prove Theorem \ref{f-thm1.6}. 

\begin{proof}[Proof of Theorem \ref{f-thm1.6}]
We divide the proof into several small steps. 
\setcounter{step}{0}
\begin{step}\label{f-4.2-step1}
If $\dim (X\setminus X_{-\infty})=0$, then 
the statement is obvious. 
From now on, we assume $n\geq 1$ and use induction on $\dim 
(X\setminus X_{-\infty})$. 
Therefore, we assume that 
the statement holds true when $\dim (X\setminus X_{-\infty})<n$. 
\end{step}
\begin{step}\label{f-4.2-step2}
Let $C$ be a qlc stratum of $[X, \omega]$. 
We put $X'=C\cup X_{-\infty}$. 
Then, by adjunction (see Theorem \ref{f-thm2.13} (i)), 
$[X', \omega|_{X'}]$ is a quasi-log scheme. 
Note that $\omega|_{X'}+r\mathcal L|_{X'}\equiv 0$ holds. 
Let $\mathcal I_{X'}$ be the defining ideal sheaf 
of $X'$ on $X$.  
By Theorem \ref{f-thm2.13} (ii), 
we have 
$H^1(X, \mathcal I_{X'}\otimes \mathcal L)=0$ since 
$\mathcal L-\omega\equiv (r+1) \mathcal L$ is ample. 
Therefore, the natural restriction 
map 
\begin{equation}\label{f-eq4.1}
H^0(X, \mathcal L)\to H^0(X', \mathcal L|_{X'}) 
\end{equation}
is surjective. 
\end{step}
\begin{step}\label{f-4.2-step3}
If $\dim C<n$, then $|\mathcal L|_{X'}|$ is basepoint-free 
by the induction hypothesis. 
By \eqref{f-eq4.1}, $|\mathcal L|$ is basepoint-free 
in a neighborhood of $X'$. 
\end{step}
\begin{step}\label{f-4.2-step4}
If $\dim C=n$ and $C\cap X_{-\infty}=\emptyset$, then 
$|\mathcal L|_C|$ is basepoint-free 
by Lemma \ref{f-lem4.1} since 
$[C, \omega|_C]$ is an irreducible 
quasi-log canonical pair with 
$$\omega|_C+r\mathcal L|_C\equiv 0$$ and 
$$r>\dim (X'\setminus 
X'_{-\infty})-1=\dim C-1. 
$$ 
We note that $|\mathcal L|_{X_{-\infty}}|$ is basepoint-free 
by assumption. 
Therefore, $|\mathcal L|_{X'}|$ is obviously basepoint-free. 
Hence, by \eqref{f-eq4.1}, $|\mathcal L|$ is basepoint-free 
in a neighborhood of $X'$. 
\end{step}
\begin{step}\label{f-4.2-step5}
By Steps \ref{f-4.2-step3}, \ref{f-4.2-step4}, and 
\eqref{f-eq4.1}, we may assume that 
$X\setminus X_{-\infty}$ is irreducible with 
$\dim (X\setminus X_{-\infty})=n$ such that 
$X$ is connected. Since $\mathcal L-\omega\equiv 
(r+1)\mathcal L$ is ample, 
$H^1(X, \mathcal I_{X_{-\infty}}\otimes \mathcal L)=0$ by 
Theorem \ref{f-thm2.13} (ii). 
Therefore, the natural restriction map 
$$
H^0(X, \mathcal L)\to H^0(X_{-\infty}, \mathcal L|_{X_{-\infty}})
$$ 
is surjective. Since $|\mathcal L|_{X_{-\infty}}|$ is basepoint-free 
by assumption, 
the base locus $\Bs\!|\mathcal L|$ of $|\mathcal L|$ is disjoint 
from $X_{-\infty}$. Since $X\setminus X_{-\infty}$ is irreducible 
and $X$ is connected, $\Bs\!|\mathcal L|$ does not contain $X\setminus 
X_{-\infty}$. 
By Step \ref{f-4.2-step3}, $\Bs\!|\mathcal L|$ contains no qlc 
centers of $[X, \omega]$. Hence 
$\Bs\!|\mathcal L|$ contains no qlc strata of $[X, \omega]$. 

We assume that $\Bs\!|\mathcal L|\ne \emptyset$. 
We take a general member $D$ of $|\mathcal L|$.  
Then we can take $0<c\leq 1$ such 
that $[X, \omega+cD]$ is a quasi-log scheme with 
$$\Nqlc (X, \omega+cD)=\Nqlc (X, \omega)$$ and 
that there exists a qlc center $C$ of $[X, \omega+cD]$ with 
$C\cap \Bs\!|\mathcal L|\ne \emptyset$ by construction (see 
Proposition \ref{f-prop3.4}). 
We put $$X'=C\cup \Nqlc (X, \omega+cD). 
$$
By adjunction (see Theorem \ref{f-thm2.13} (i)), 
$[X', (\omega+cD)|_{X'}]$ is a quasi-log scheme. 
By construction, $\dim C<n$ and 
$$(\omega+cD)|_{X'}+(r-c)\mathcal L|_{X'}\equiv 0$$ hold. 
Note that $$r-c> \dim C-1=\dim (X'\setminus 
X'_{-\infty})-1$$ holds. 
Therefore, by the induction hypothesis, 
$|\mathcal L|_{X'}|$ is basepoint-free. 
Since $\mathcal L-(\omega+cD)\equiv (r+1-c)\mathcal L$ 
is ample, 
$H^1(X, \mathcal I_{X'}\otimes \mathcal L)=0$ 
by Theorem \ref{f-thm2.13} (ii), 
where $\mathcal I_{X'}$ is the defining ideal sheaf 
of $X'$ on $X$. 
Thus, the restriction map 
$$
H^0(X, \mathcal L)\to H^0(X', \mathcal L|_{X'}) 
$$ 
is surjective. In particular, 
$|\mathcal L|$ is basepoint-free in a neighborhood 
of $C$. This is a contradiction since $C\cap \Bs\!|\mathcal L|\ne 
\emptyset$. 
Hence, we obtain $\Bs\!|\mathcal L|=\emptyset$. 
\end{step} 
We obtained the desired statement. 
\end{proof}

\section{Proof of Theorem \ref{f-thm1.1}}\label{f-sec5} 
Let us explain the idea of the proof of the relative spannedness 
in Theorem \ref{f-thm1.1}. 
We construct a sequence of closed subschemes 
$$
Z_0\subset Z_1\subset \cdots \subset Z_{k-1}\subset X
$$ 
such that $Z_{k-1}=F$ holds set theoretically. 
For every $i$, there exists an $\mathbb R$-Cartier divisor 
$\omega_i|_{Z_i}$ on $Z_i$ such that 
$[Z_i, \omega_i|_{Z_i}]$ is a quasi-log scheme and 
that $$\omega_i|_{Z_i}+rL|_{Z_i}\equiv 0$$ holds with 
$$r>\dim F-1\geq \dim Z_i-1. 
$$ 
We can make $[Z_i, \omega_i|_{Z_i}]$ satisfy that 
$(Z_0)_{-\infty}=\emptyset$ and that 
$(Z_{i+1})_{-\infty}\subset Z_i$ set theoretically for every $i$. 
By Theorem \ref{f-thm1.6}, the complete linear system 
$|L|_{Z_0}|$ is basepoint-free. 
By the vanishing theorem for quasi-log schemes, 
we obtain that the natural restriction map 
$$
f_*\mathcal O_X(L)\to f_*\mathcal O_{Z_i}(L|_{Z_i})
$$ 
is surjective for every $i$. 
Therefore, if $|L|_{Z_i}|$ is basepoint-free, 
then the relative base locus $\Bs_f\!|L|$ is disjoint from 
$Z_i$. 
This implies that $|L|_{(Z_{i+1})_{-\infty}}|$ is basepoint-free 
since we have $(Z_{i+1})_{-\infty}\subset Z_i$. 
Then, by Theorem \ref{f-thm1.6}, 
the complete linear system $|L|_{Z_{i+1}}|$ is basepoint-free. 
Hence we finally obtain that the relative base locus 
$\Bs_f\!|L|$ is disjoint from $F$. 
This means that 
$$
f^*f_*\mathcal O_X(L)\to \mathcal O_X(L)
$$ 
is surjective at every point of $F$. 

\medskip 

We start with the following 
easy lemma on log canonical pairs. 

\begin{lem}\label{f-lem5.1}
Let $(X, \Delta)$ be a log canonical pair and 
let $B$ be an effective $\mathbb R$-Cartier divisor 
on $X$ such that 
$(X, \Delta+B)$ is not log canonical. 
Then there exists an increasing sequence of 
real numbers 
$$
0\leq c_0< c_1< \cdots <c_{k-1}<c_k=1
$$
with the following properties. 
\begin{itemize}
\item[(i)] $c_0$ is the log canonical threshold 
of $(X, \Delta)$ with respect to $B$. 
\item[(ii)] We put $U_i=X\setminus \Nlc(X, 
\Delta+c_iB)$ 
for every $i$. 
Then $U_{i+1}\subsetneq U_i$ holds 
for every $0\leq i\leq k-1$. 
\item[(iii)] For every $1\leq i\leq k$, 
$X\setminus \Nlc (X, \Delta+tB)=U_i$ holds for any 
$t\in (c_{i-1}, c_i]$. 
\end{itemize}
In this situation, for each $i$ with 
$0\leq i\leq k-1$, there exists 
a finite set of log canonical centers $\{C_j\}_{j\in I_i}$ of 
$(X, \Delta+c_iB)$ such that 
$$
U_i\setminus U_{i+1}\subset \bigcup_{j\in I_i} C_j
$$ 
and that 
$$
\Nlc(X, \Delta+B)=\bigcup_{i=0}^{k-1}
\left(\bigcup_{j\in I_i} C_j\right)
$$ 
holds set theoretically. 
\end{lem}

\begin{proof} 
We note that $c_i$ is a kind of jumping numbers 
of $(X, \Delta)$ with 
respect to $B$ for every $i$. 
More precisely, we consider the following Zariski open set 
$$
U_t:=X\setminus \Nlc(X, \Delta+tB)
$$ for every $t\in [0, 1]$ and increase $t$ from $0$ to $1$. 
Then there exists an increasing sequence of real numbers 
$$
0\leq c_0< c_1< \cdots <c_{k-1}<c_k=1
$$
satisfying the desired properties. 

The following description may be helpful. 
By the above construction, $c_0$ is the log canonical threshold of 
$(X, \Delta)$ with respect to $B$ and $c_k=1$. 
Let $\Nklt(X, \Delta+c_{i-1}B)$ denote the non-klt locus of 
$(X, \Delta+c_{i-1}B)$ for $1\leq i\leq k-1$. Equivalently, 
$V_{i-1}:=X\setminus \Nklt(X, \Delta+c_{i-1}B)$ is the largest 
Zariski open set of $X$ such that 
$(V_{i-1}, (\Delta+c_{i-1}B)|_{V_{i-1}})$ is kawamata log terminal. 
Then $c_i-c_{i-1}$ is the log canonical threshold of $(X, \Delta+c_{i-1}B)$ 
with respect to $B$ on the Zariski open set 
$V_{i-1}=X\setminus \Nklt(X, \Delta+c_{i-1}B)$ for 
$1\leq i\leq k-1$. 
\end{proof}

We prepare one more easy lemma. 

\begin{lem}\label{f-lem5.2}
Let $(X, \Delta)$ be a log canonical pair 
and let $B_1, \ldots, B_k$ be effective Cartier divisors 
on $X$ passing through a closed point $P$ of $X$. 
If $(X, \Delta+\sum _{i=1}^kB_i)$ is log 
canonical around $P$, then the inequality $k\leq \dim X$ holds. 
\end{lem}

Although Lemma \ref{f-lem5.2} is well known, 
we prove it here for the reader's convenience. 

\begin{proof}
By shrinking $X$ around $P$, 
we may assume that $(X, \Delta+\sum _{i=1}^k B_i)$ is 
log canonical. If $\dim X=1$, then 
the statement is obvious. 
We use the induction on $\dim X$. 
So we assume that 
$\dim X\geq 2$ holds. 
Let $\nu\colon  Z\to B_k$ be the normalization of $B_k$. 
We put 
$$
K_Z+\Delta_Z=\nu ^* (K_X+\Delta+B_k). 
$$ 
Then $(Z, \Delta_Z)$ is log canonical 
by adjunction since $(X, \Delta+B_k)$ is 
log canonical. 
We note that $\Supp\!B_i$ and 
$\Supp\!B_k$ have no common irreducible components 
for $1\leq i\leq k-1$ since $(X, \Delta+\sum _{i=1}^kB_i)$ is 
log canonical. 
We take $Q\in \nu^{-1}(P)$. 
Then $(Z, \Delta_Z+\sum _{i=1}^{k-1}\nu^*B_i)$ is log canonical 
by adjunction and $Q\in \Supp\nu^*B_i$ for 
$1\leq i\leq k-1$. 
Therefore, we obtain 
$$
k-1\leq \dim Z=\dim X-1
$$ 
by the induction hypothesis. 
This means that the desired inequality $k\leq \dim X$ holds. 
\end{proof}

Let us prove Theorem \ref{f-thm1.1} by using 
Theorem \ref{f-thm1.6} and Lemma \ref{f-lem5.1}. 

\begin{proof}[Proof of Theorem \ref{f-thm1.1}]
Since $K_X+\Delta+rL\equiv _Y 0$, 
$-(K_X+\Delta)$ is $f$-ample, 
and $r>0$, 
we see that $L$ is $f$-ample. 
We put $f(F)=P$ and shrink $Y$ around 
$P$. 
Then we may assume that $Y$ is affine without loss of generality. 
We put $n=\dim X$ and 
take general hyperplane sections $B_1, \ldots, B_{n+1}$ 
on $Y$ such that $P\in \Supp\!B_i$ for every $i$. 
We put $$B=\sum _{i=1}^{n+1} f^*B_i. 
$$ 
Then $(X, \Delta+B)$ is log canonical 
outside $F$ and is not log canonical 
at every point of $F$ by Lemma \ref{f-lem5.2}. 
\setcounter{step}{0}
\begin{step}\label{f-5.2-step1}
Let $F'$ be any positive-dimensional irreducible component of 
$F$. 
In this step, we will prove that 
$\dim F'\geq r-1$ holds. 

We put 
$$
c=\max \{ t\in \mathbb R\, |\, \text{$(X, \Delta+tB)$ is 
log canonical at the generic point of $F'$}\},  
$$ 
that is, $c$ is the log canonical threshold of $(X, \Delta)$ with 
respect to $B$ at the generic point of $F'$. 
By construction, $0\leq c<1$ and $F'$ is a log canonical 
center of $(X, \Delta+cB)$. 
We now consider the natural quasi-log scheme structure 
of $[X, \Delta+cB]$ 
as in Example \ref{f-ex2.14}. 
We put 
$$
X'=F'\cup \Nqlc(X, \Delta+cB) 
$$ 
and consider the induced 
quasi-log scheme $[X', (K_X+\Delta+cB)|_{X'}]$ 
by adjunction (see Theorem \ref{f-thm2.13} (i)). 
Note that 
$$
tL|_{X'}-(K_X+\Delta +cB)|_{X'}\equiv (t+r)L|_{X'}
$$ 
is ample for $t>-r$ since $f(X')=P$. 
We note that 
$$
\deg \chi (X', \mathcal I_{X'_{-\infty}}\otimes 
\mathcal O_{X'}(tL))=\dim F'
$$ 
holds because $L|_{X'}$ is ample and 
the coherent ideal sheaf $\mathcal I_{X'_{-\infty}}$ on $X'$ can 
be considered a coherent sheaf on $F'$. 
More precisely,  
$\mathcal I_{X'_{-\infty}}\subset \mathcal O_{F'}$ holds 
since $\{0\}=\mathcal I_{F'}\cap \mathcal I_{X'_{-\infty}}\subset 
\mathcal O_{X'}$, where $\mathcal I_{F'}$ is the defining ideal sheaf of 
$F'$ on $X'$. 
By Theorem \ref{f-thm2.13}, 
$$
H^i(X', \mathcal I_{X'_{-\infty}}\otimes 
\mathcal O_{X'}(tL))=0
$$ 
for $i>0$ and $t\in \mathbb Z$ with 
$t>-r$. 
Since $L|_{X'}$ is ample, 
$$
H^0(X', \mathcal I_{X'_{-\infty}}\otimes \mathcal O_{X'}
(tL))=0
$$ 
for $t\in \mathbb Z$ with $t<0$. 
Therefore, we obtain 
$$
\chi (X', \mathcal I_{X'_{-\infty}}\otimes 
\mathcal O_{X'}(tL))=0
$$ 
for $t\in \mathbb Z$ with 
$-r<t\leq -1$. Hence, we obtain 
$$
\dim F'=\deg \chi (X', \mathcal I_{X'_{-\infty}}\otimes 
\mathcal O_{X'}(tL))\geq r-1. 
$$ 

This means that the dimension of every positive-dimensional 
irreducible 
component of $F$ is $\geq r-1$. 
\end{step}
\begin{step}\label{f-5.2-step2}
In Steps \ref{f-5.2-step2} and \ref{f-5.2-step3}, 
we will prove that $f^*f_*\mathcal O_X(L)\to \mathcal O_X(L)$ 
is surjective at every point of $F$. 

By Lemma \ref{f-lem5.1}, we have an increasing sequence 
of real numbers $$0\leq c_0<c_1<\cdots <c_k=1 
$$ 
satisfying the properties in Lemma \ref{f-lem5.1}. 
We consider normal pairs $(X, \Delta+c_iB)$ for $0\leq i\leq k-1$. 
We put $\omega_i=K_X+\Delta+c_iB$. 
Then $[X, \omega_i]$ is a quasi-log scheme for $0\leq i\leq k-1$ 
(see Example \ref{f-ex2.14}). 
We put 
$$
Z_i=\bigcup _{j\in I_i}C_j\cup \Nqlc (X, \omega_i)
$$ 
and consider the pair $[Z_i, \omega_i|_{Z_i}]$ for every $i$ with 
$0\leq i\leq k-1$. 
Then, by adjunction (see Theorem \ref{f-thm2.13} (i)), 
$[Z_i, \omega_i|_{Z_i}]$ is a quasi-log scheme with 
$f(Z_i)=P$ for $0\leq i\leq k-1$. 
We note that 
$\Nqlc (X, \omega_0)=\emptyset$ since 
$(X, \Delta+c_0B)$ is log canonical by definition. 
We also note that $(Z_i)_{-\infty}=\Nqlc (Z_i, \omega_i|_{Z_i})
=\Nqlc (X, \omega_i)$ for 
every $i$ by 
construction. 
Since $L-\omega_i$ is numerically equivalent to $$L-(K_X+\Delta)
\equiv_Y (r+1)L$$ over $Y$, 
$L-\omega_i$ is $f$-ample. 
Therefore, by Theorem \ref{f-thm2.13} (ii), 
$$H^1(X, \mathcal I_{Z_i}\otimes \mathcal O_X(L))=0, 
$$ 
where $\mathcal I_{Z_i}$ is the defining ideal sheaf of 
$Z_i$ on $X$. 
Hence, the restriction map 
\begin{equation}\label{f-eq5.1}
H^0(X, \mathcal O_X(L))\to H^0(Z_i, \mathcal O_{Z_i}(L))
\end{equation} 
is surjective for every $0\leq i\leq k-1$. 
\end{step}
\begin{step}\label{f-5.2-step3}
Since $[Z_0, \omega_0|_{Z_0}]$ is a projective quasi-log canonical 
pair such that 
$$\omega_0|_{Z_0}+rL|_{Z_0}\equiv 0$$ with 
$r>\dim F-1\geq \dim Z_0-1$, the complete linear system 
$|L|_{Z_0}|$ is basepoint-free by Corollary 
\ref{f-cor4.3}. 

If $|L|_{Z_i}|$ is basepoint-free, 
then the relative base locus $\Bs_f\!|L|$ is disjoint 
from $Z_i$ by \eqref{f-eq5.1}. 
By Lemma \ref{f-lem5.1}, $\Nqlc(X, \omega_{i+1})\subset 
Z_i$ holds set theoretically. 
This implies that $\Bs_f\!|L|$ does not intersect with 
$\Nqlc(X, \omega_{i+1})=\Nqlc(Z_{i+1}, \omega_{i+1}|_{Z_{i+1}})=
(Z_{i+1})_{-\infty}$. 
Therefore, $|L|_{(Z_{i+1})_{-\infty}}|$ is basepoint-free. 
Since 
$$
\omega_{i+1}|_{Z_{i+1}}+rL|_{Z_{i+1}}\equiv 0
$$ 
with $r>\dim F-1\geq \dim Z_{i+1}-1$, 
$|L|_{Z_{i+1}}|$ is basepoint-free by Theorem 
\ref{f-thm1.6}. 
We repeat this process. 
We note that $F=\Nlc(X, \Delta+B)=\Nqlc(X, \omega_k)$ set theoretically. 
Hence we finally obtain that the complete linear system 
$|L|_{Z_{k-1}}|$ is basepoint-free and that the 
relative base locus $\Bs_f\!|L|$ is disjoint from $F=\Nqlc(X, \omega_k)$, 
equivalently, $f^*f_*\mathcal O_X(L)\to \mathcal O_X(L)$ is 
surjective at every point of $F$. 
\end{step}
We obtained all the desired statements. 
\end{proof}

\begin{proof}[Proof of Corollary \ref{f-cor1.2}]
We assume that $K_X+\Delta+(n+1)L$ is not $f$-nef. 
Then, by the cone and contraction theorem for log canonical 
pairs (see \cite[Theorem 1.1]{fujino-fund}), 
we get a $(K_X+\Delta+(n+1)L)$-negative 
extremal contraction $\varphi\colon X\to W$ over $Y$. 
Thus, by replacing $f\colon X\to Y$ with $\varphi\colon X\to W$, we may assume that 
the relative Picard number $\rho(X/Y)=1$. 
Therefore, there exists $r$ with 
$r>n+1$ such that $K_X+\Delta+rL$ is relatively numerically trivial over $Y$. 
By Theorem \ref{f-thm1.1}, 
we have 
$$
n= \dim X\geq r-1>n. 
$$ 
This is a contradiction. This means that 
$K_X+\Delta+(n+1)L$ is $f$-nef. 
Similarly, we can check that 
$K_X+\Delta+nL$ is $f$-nef when $\dim Y\geq 1$. 
\end{proof}

We close this section with the proof of 
Corollary \ref{p-cor1.3}. 

\begin{proof}[Proof of Corollary \ref{p-cor1.3}]
Without loss of generality, we may assume that $Y$ is affine 
by shrinking $Y$ around $f(F)$. 
Since $\dim X=\dim Y$ and $f$ is surjective, 
$f$ is generically finite. 
Hence $f_*\mathcal O_X(-L)\ne 0$ holds. Thus 
we can take an effective Cartier divisor $D$ on $X$ 
such that $D\sim -L$. 
Since $(X, \Delta)$ is kawamata log terminal, $(X, \Delta
+\varepsilon D)$ is also kawamata log terminal 
for $0<\varepsilon \ll 1$. 
By construction, 
$$K_X+\Delta+\varepsilon D+(r+\varepsilon)L$$ 
is relatively numerically trivial over $Y$. 
Therefore, by Theorem \ref{f-thm1.1}, 
the dimension of every positive-dimensional irreducible 
component of $F$ is $\geq (r+\varepsilon) -1$, 
that is, $\geq \lfloor r\rfloor$. 
If $\dim F\leq r+1$, then $\dim F<(r+\varepsilon)+1$ obviouly 
holds. Thus, by Theorem \ref{f-thm1.1}, 
$$f^*f_*\mathcal O_X(L)\to \mathcal O_X(L)$$ is surjective 
at every point of $F$. 
\end{proof}

\section{Generalizations for quasi-log canonical pairs}\label{f-sec6}

In this section, we will prove Theorem \ref{f-thm1.7}. 
The following lemma is a generalization of Lemma \ref{f-lem5.1} 
for quasi-log canonical pairs. 

\begin{lem}\label{f-lem6.1}
Let $[X, \omega]$ be an irreducible quasi-log canonical 
pair and let $B$ be an effective 
$\mathbb R$-Cartier divisor on $X$. 
Then there exist an increasing sequence of real numbers 
$$
c_{-1}=0\leq c_0< c_1< \cdots <c_{k-1}<c_k=1, 
$$ 
globally embedded simple normal crossing pairs 
$(Y_i, B_{Y_i})$ for $0\leq i \leq k$, and 
proper surjective morphisms $f_i\colon  Y_i\to X$ for 
all $0\leq i\leq k$ with the following properties. 
\begin{itemize}
\item[(i)] For every $0\leq i\leq k$, 
$$
\left(X, \omega+c_iB, f_i\colon  (Y_i, B_{Y_i})\to X\right)
$$ 
is a quasi-log scheme. 
\item[(ii)] We put $$U_i=X\setminus \Nqlc(X, \omega+c_iB)$$ for 
every $0\leq i\leq k$.  
Then 
$$
U_k\subsetneq U_{k-1}\subsetneq \cdots \subsetneq U_0=X 
$$ 
holds. 
\item[(iii)] For every $0\leq i\leq k$, 
$$
\left(X, \omega+tB, f_i\colon  (Y_i, B_{Y_i}+(t-c_i)f_i^*B)
\to X\right)
$$ 
is a quasi-log scheme such that 
$$
U_i=X\setminus \Nqlc (X, \omega+tB)
$$ 
holds for any $t\in (c_{i-1}, c_i]$. 
\item[(iv)] For each $0\leq i\leq k-1$, there exists a 
finite set of qlc centers $\{C_j\}_{j\in I_i}$ of 
$[X, \omega+c_iB]$ such that 
$$
U_i\setminus U_{i+1} \subset \bigcup _{j\in I_i}C_j
$$ 
holds. 
\end{itemize}
\end{lem}

Before we prove Lemma \ref{f-lem6.1}, we make an important remark. 

\begin{rem}\label{f-rem6.2} 
In Lemma \ref{f-lem6.1}, 
let $f\colon (Y, B_Y)\to X$ be a quasi-log resolution of $[X, \omega]$. 
Let $X'$ be the union of all qlc centers of $[X, \omega]$ contained 
in $\Supp B$. 
Assume that the union of all strata of $(Y, B_Y)$ mapped to $X'$ by $f$, 
which is denoted by $Y'$, is a union of 
some irreducible components of $Y$. 
We put 
$Y''=Y-Y'$, $K_{Y''}+B_{Y''}=(K_Y+B_Y)|_{Y''}$, 
and $f''=f|_{Y''}$. 
By \cite[Proposition 6.3.1]{fujino-foundations} and 
\cite[Theorem 3.35]{kollar}, we may further assume that 
$$
(Y'', B_{Y''}+(f'')^*B)
$$ 
is a globally embedded simple normal crossing pair. 
Then, by Lemma \ref{f-lem3.2}, 
$$
\left(X, \omega+B, f''\colon  (Y'', B_{Y''}+(f'')^*B)\to X\right)
$$ 
is a quasi-log scheme. 
By the following proof of Lemma \ref{f-lem6.1}, 
we see that 
$$
\Nqlc(X, \omega+B)=\bigcup_{i=0}^{k-1}\left(\bigcup _{j\in I_i} C_j\right)
$$ 
holds set theoretically. 
\end{rem}

We give a detailed proof of Lemma \ref{f-lem6.1} for 
the reader's convenience, although it is similar to 
the proof of Lemma \ref{f-lem5.1}.  

\begin{proof}[Proof of Lemma \ref{f-lem6.1}]
\setcounter{step}{0}
Let $f\colon (Y, B_Y)\to X$ be a quasi-log resolution of $[X, \omega]$. 
\begin{step}\label{f-6.1-step1} 
If there exists a qlc center $C$ of $[X, \omega]$ such that 
$C\subset \Supp\!B$. 
Then we put $c_0=0$, $(Y_0, B_{Y_0})=(Y, B_Y)$, and $f_0=f$. 
\end{step}
\begin{step}\label{f-6.1-step2} 
We assume that there are no qlc centers of $[X, \omega]$ contained 
in $\Supp\!B$. 
By \cite[Proposition 6.3.1]{fujino-foundations} and 
\cite[Theorem 3.35]{kollar}, we may assume that 
$$
\left(Y, f^*B+\Supp\!B_Y\right)
$$ 
is a globally embedded simple normal crossing pair. 

If $(B_Y+f^*B)^{>1}=0$, then we put $c_0=1$, 
$(Y_0, B_{Y_0})=(Y, B_Y+f^*B)$, $f_0=f$, and 
we stop this process (see Lemma \ref{f-lem3.1}). 

If $(B_Y+f^*B)^{>1}\ne 0$, then we can take $0<c_0<1$ such that 
$(B_Y+c_0f^*B)^{>1}=0$ and that 
there exists a component $G$ of $(B_Y+c_0f^*B)^{=1}$ 
with $f(G)\subset \Supp\!B$. 
In this situation, we put $(Y_0, B_{Y_0})=(Y, B_Y+c_0f^*B)$ and 
$f_0=f$. 
Then we see that 
$$
\left(X, \omega+c_0B, f_0\colon  (Y_0, B_{Y_0})\to X\right) 
$$ 
is the desired quasi-log canonical pair (see 
Lemma \ref{f-lem3.1}). 
\end{step}
\begin{step}\label{f-6.1-step3}
We assume that we have already constructed 
$$
\left(X, \omega+c_iB, f_i\colon  (Y_i, B_{Y_i})\to X\right)
$$ 
for $i\geq 0$ with $c_i<1$. 

Let $X'_i$ be the union of $\Nqlc (X, \omega+c_iB)$ and all qlc centers of 
$[X, \omega+c_iB]$ contained in $\Supp\!B$. 
By \cite[Proposition 6.3.1]{fujino-foundations}, 
we may assume that the union of all strata of $(Y_i, B_{Y_i})$ mapped to $X'_i$ 
by $f_i$, which is denoted by $Y'_i$, is a union of 
some irreducible components of $Y_i$. 
We put $Y''_i=Y_i-Y'_i$, $$K_{Y''_i}+B_{Y''_i}
=(K_{Y_i}+B_{Y_i})|_{Y''_i},$$ and 
$f''_i=f_i|_{Y''_i}$. 
We may further assume that 
$$
\left(Y''_i, (f''_i)^*B+\Supp\!B_{Y''_i}\right)
$$ 
is a globally embedded simple normal crossing pair 
by \cite[Proposition 6.3.1]{fujino-foundations} and 
\cite[Theorem 3.35]{kollar}. 

If 
$$
f''_i\left(\Supp\!\left(B_{Y''_i}+(1-c_i)(f''_i)^*B\right)^{>1}\right)\subset 
X'_i
$$ 
holds, then we put 
$c_{i+1}=1$, 
$$
\left(Y_{i+1}, B_{Y_{i+1}}\right)=\left(
Y''_i, B_{Y''_i}+(1-c_i)(f''_i)^*B\right), 
$$ 
$f_{i+1}=f''_i$, and we stop this process. 
We can see that 
$$
\left(X, \omega+B, f_{i+1}\colon  (Y_{i+1}, B_{Y_{i+1}})\to X\right)
$$ 
with $c_{i+1}=1$ is a quasi-log scheme with the desired 
properties (see Lemma \ref{f-lem3.2}). 

Otherwise, we put 
$$
c_{i+1}=\sup \left\{s\in \mathbb R \, \left|\, 
f''_i\left(\Supp\!\left(B_{Y''_i}+(s-c_i)(f''_i)B\right)^{>1}\right)\subset 
X'_i\right.\right\}. 
$$ In this situation, 
we have $c_i<c_{i+1}<1$. 
Then we put 
$$
\left(Y_{i+1}, B_{Y_{i+1}}\right)=\left(
Y''_i, B_{Y''_i}+(c_{i+1}-c_i)(f''_i)^*B\right) 
$$ and $f_{i+1}=f''_i$. 
We can see that 
$$
\left(X, \omega+c_{i+1}B, f_{i+1}\colon  (Y_{i+1}, B_{Y_{i+1}})\to X\right)
$$ 
is a quasi-log scheme with the desired 
properties (see Lemma \ref{f-lem3.2}). 

\end{step}
\begin{step}\label{f-6.1-step4}
After finitely many steps, we get a finite increasing 
sequence of real numbers: 
$$
c_{-1}=0\leq c_0< c_1< \cdots <c_{k-1}<c_k=1. 
$$ 
By the above construction, we obviously have the desired 
properties. 
\end{step}
Roughly speaking, $c_i-c_{i-1}$ is the quasi-log canonical threshold 
of $[X, \omega+c_{i-1}B]$ with respect to $B$ on 
the Zariski open set $X\setminus X'_{i-1}$ of $X$ 
for $1\leq i\leq k-1$. Hence we can see this lemma as a quasi-log 
scheme analogue of Lemma \ref{f-lem5.1}. 
\end{proof}

We give a sketch of the proof of Theorem \ref{f-thm1.7} 
for the reader's convenience, although the 
proof of Theorem \ref{f-thm1.7} is essentially 
the same as that of Theorem \ref{f-thm1.1}. 

\begin{proof}[Sketch of Proof of Theorem \ref{f-thm1.7}]
We divide the proof into several small steps. 
\setcounter{step}{0}
\begin{step}\label{f-1.7-step1}
Since $\mathcal L-\omega\equiv (r+1)\mathcal L$ is $\varphi$-ample, 
we have $R^1\varphi_*(\mathcal I_{X_i}\otimes \mathcal L)=0$, 
where $X_i$ is any irreducible component of $X$ and 
$\mathcal I_{X_i}$ is the defining ideal sheaf of $X_i$ on $X$. 
Therefore, the restriction map 
$$
\varphi_*\mathcal L\to \varphi_*(\mathcal L|_{X_i})
$$ 
is surjective. We note that $[X_i, \omega|_{X_i}]$ is 
a quasi-log canonical pair by adjunction (see Theorem \ref{f-thm2.13} (i)). 
We also note that $\omega|_{X_i}+r\mathcal 
L|_{X_i}$ is relatively numerically trivial over $W$. 
Therefore, by replacing $[X, \omega]$ with $[X_i, \omega|_{X_i}]$, 
we may assume that $X$ is irreducible. 
Furthermore, by replacing $W$ with $\varphi(X)$, 
we may assume that $W$ is an irreducible variety. 
By shrinking $W$ around $\varphi(F)$, we may further assume 
that $W$ is an affine variety. 
\end{step}
\begin{step}\label{f-step1.7-step2}
If $\varphi(X)=W$ is a point, 
then we have $X=F$ and may assume that 
$\dim X=\dim F\geq 1$ holds. 
In this case, by Lemma \ref{f-lem4.1}, $r\leq \dim F+1$ holds. 
This means that $\dim F\geq r-1$ holds true. 
If $\dim F<r+1$, equivalently, $r>\dim F-1=\dim X-1$, 
then $|\mathcal L|$ is basepoint-free by Lemma \ref{f-lem4.1} again. 
This means that 
$$\varphi^*\varphi_*\mathcal L\to \mathcal L$$ is surjective 
at every point of $F$. 
\end{step}
\begin{step}\label{f-1.7-step3}
From now on, we may assume that 
$\dim W\geq 1$ holds. We put $n=\dim X$ and take general 
hyperplane sections 
$B_1, \ldots, B_{n+1}$ on $W$ such that 
$\varphi(F)\in \Supp\!B_i$ for every $i$. 
We put 
$$
B=\sum _{i=1}^{n+1} \varphi^*B_i. 
$$
\end{step}
\begin{step}\label{f-1.7-step4} 
Let $F'$ be any positive-dimensional irreducible component of $F$. 

If $F'$ is a qlc center of $[X, \omega]$. 
Then $[F', \omega|_{F'}]$ is a quasi-log canonical 
pair by adjunction (see Theorem \ref{f-thm2.13} (i)). Hence we obtain 
$$
\dim F'=\deg \chi (F', \mathcal L^{\otimes t}|_{F'})\geq r-1
$$ 
by the usual application of the vanishing theorem 
(see Theorem \ref{f-thm2.13} (ii) and 
Step \ref{f-4.1-step1} in the proof of Lemma \ref{f-lem4.1}).  

From now on, we may assume that $F'$ is not a qlc center 
of $[X, \omega]$. 
Let $f\colon (Y, B_Y)\to X$ be a quasi-log resolution of $[X, \omega]$. 
Let $X'$ be the union of 
all qlc centers contained in $F$. 
By \cite[Proposition 6.3.1]{fujino-foundations}, 
we may assume that 
the union of all strata of $(Y, B_Y)$ mapped to 
$X'$ by $f$, which is denoted by $Y'$, is a union of 
some irreducible components of $Y$. 
We put $Y''=Y-Y'$, $K_{Y''}+B_{Y''}=(K_Y+B_Y)|_{Y''}$, 
and $f''=f|_{Y''}$. 
We may further assume that 
$$
\left(Y'', (f'')^*B +\Supp\!B_{Y''}\right)
$$ 
is a globally embedded simple normal crossing pair by 
\cite[Proposition 6.3.1]{fujino-foundations} and 
\cite[Theorem 3.35]{kollar}. 
By \cite[Lemma 6.3.13]{fujino-foundations}, we can take 
$0<c<1$ such that 
there exists an irreducible component 
$G$ of $(B_{Y''}+c(f'')^*B)^{=1}$ with 
$f''(G)=F'$ and 
that $F'\not\subset f''\left(\Supp\!\left(B_{Y''}
+c(f'')^*B\right)^{>1}\right)$. 
Then 
$$
\left(X, \omega+cB, f''\colon  \left(Y'', B_{Y''}+c(f'')^*B\right)
\to X\right)
$$ 
is a quasi-log scheme such that 
$F'$ is a qlc center of $[X, \omega+cB]$ (see 
Lemma \ref{f-lem3.2}). 
We put 
$$
X'=F'\cup \Nqlc (X, \omega+cB). 
$$ 
Then, by adjunction (see Theorem \ref{f-thm2.13} (i)), 
$[X', (\omega+cB)|_{X'}]$ is a quasi-log scheme. 
By construction, 
$$
\dim F'=\deg \chi (X', \mathcal I_{X'_{-\infty}}\otimes 
\mathcal L^{\otimes t})\geq r-1
$$ 
as in Step \ref{f-5.2-step1} in the proof of Theorem \ref{f-thm1.1}. 
\end{step}
\begin{step}\label{f-1.7-step5}
We note that 
$$
\left(X, \omega+B, f''\colon  (Y'', B_{Y''}+(f'')^*B)\to X\right)
$$ 
is a quasi-log scheme (see Lemma \ref{f-lem3.2}) 
such that $\Nqlc(X, \omega+B)=F$ holds set theoretically 
(see \cite[Lemma 6.3.13]{fujino-foundations}). 
By Lemma \ref{f-lem6.1}, 
the arguments in Steps \ref{f-5.2-step2} and \ref{f-5.2-step3} in the 
proof of Theorem \ref{f-thm1.1} work with some minor modifications. 
Hence, we obtain 
that 
$$
\varphi^*\varphi_*\mathcal L\to \mathcal L
$$ 
is surjective at every point of $F=\Nqlc(X, \omega+B)$. 
\end{step} 
We obtained all the desired statements. 
\end{proof}

\begin{proof}[Proof of Corollary \ref{f-cor1.8}]
We note that the cone and contraction theorem holds 
true for quasi-log canonical pairs (see 
\cite[Theorems 6.7.3 and 6.7.4]{fujino-foundations}). 
Therefore, the proof of Corollary \ref{f-cor1.2} works 
as well in this case, 
using Theorem \ref{f-thm1.7}.  
\end{proof}


\end{document}